\documentclass[a4paper]{amsart}

\usepackage[mathscr]{eucal}
\usepackage{amsmath}
\usepackage{amsfonts}
\usepackage{amssymb}
\usepackage{footnpag}
\usepackage[dvips]{graphicx}

\theoremstyle{plain}
\newtheorem{df}{Definition}[section]
\newtheorem{thm}{Theorem}[section]
\newtheorem{prop}{Proposition}[section]
\newtheorem{lem}{Lemma}[section]
\newtheorem{rem}{Remark}[section]
\newtheorem{ex}{Example}[section]

\newcommand{\Z}{\mathbb{Z}}

\newcommand{\C}{\mathbb{C}}
\newcommand{\R}{\mathbb{R}}

\begin{document}

\title{Toy models for D. H. Lehmer's conjecture}
\author{Eiichi Bannai and Tsuyoshi Miezaki}

\maketitle \vspace{-0.2in}
\begin{center}
Graduate School of Mathematics Kyushu University\\
Hakozaki 6-10-1 Higashi-ku, Fukuoka, 812-8581 Japan\\ \quad
\end{center} \vspace{0.1in}

\begin{quote}
{\small\bfseries Abstract.}

In 1947, Lehmer conjectured that the Ramanujan $\tau$-function $\tau (m)$ never vanishes for all positive integers $m$, where the $\tau (m)$ are the Fourier coefficients of the cusp form $\Delta _{24}$ of weight $12$. Lehmer verified the conjecture in 1947 for $m<214928639999$. In 1973, Serre verified up to $m<10^{15}$, and in 1999, Jordan and Kelly for $m<22689242781695999$. 

The theory of spherical $t$-design, and in particular those which are the shells of Euclidean lattices, is closely related to the theory of modular forms, as first shown by Venkov in 1984. In particular, Ramanujan's $\tau$-function gives the coefficients of a weighted theta series of the $E_{8}$-lattice. It is shown, by Venkov, de la Harpe, and Pache, that $\tau (m)=0$ is equivalent to the fact that the shell of norm $2m$ of the $E_{8}$-lattice is an $8$-design. So, Lehmer's conjecture is reformulated in terms of spherical $t$-design. 

Lehmer's conjecture is difficult to prove, and still remains open. In this paper, we consider toy models of Lehmer's conjecture. Namely, we show that the $m$-th Fourier coefficient of the weighted theta series of the $\mathbb{Z}^2$-lattice and the $A_{2}$-lattice does not vanish, when the shell of norm $m$ of those lattices is not the empty set. In other words, the spherical $5$ (resp. $7$)-design does not exist among the shells in the $\mathbb{Z}^2$-lattice (resp. $A_{2}$-lattice).

\vspace{-0.15in}

\quad
\end{quote}

\section{Introduction}                                   
The concept of a spherical $t$-design is due to Delsarte-Goethals-Seidel \cite{1}. For a positive integer $t$, a finite nonempty set X in the unit sphere
\[
S^{n-1} = \{x = (x_1, x_2, \cdots , x_n) \in \R ^{n}\ |\ x_1^{2}+ x_2^2+ \cdots + x_n^{2} = 1\}
\]
is called a spherical $t$-design in $S^{n-1}$ if the following condition is satisfied:
\[
\frac{1}{|X|}\sum_{x\in X}f(x)=\frac{1}{|S^{n-1}|}\int_{S^{n-1}}f(x)d\sigma (x), 
\]
for all polynomials $f(x) = f(x_1, x_2, \cdots ,x_n)$ of degree not exceeding $t$. Here, the righthand side means the surface integral on the sphere, and $|S^{n-1}|$ denotes the volume of the sphere $S^{n-1}$. The meaning of spherical $t$-designs is that the average value of the integral of any polynomial of degree up to $t$ on the sphere is replaced by the average value of a finite set on the sphere. 

Here, we denote by ${\rm {\rm Harm}}_{j}(\R^{n})$ the set of homogeneous polynomials on $\R^{n}$. It is well known that $X$ is a spherical $t$-design if and only if the condition 
\begin{align*}
\sum_{x\in X}P(x)=0,  \quad \forall P\in {\rm Harm}_{j}(\R^{n})
\end{align*}
holds for every integer $j$ with $1\leq j\leq t$. If the set $X$ is antipodal, that is $-X=X$, and $j$ is odd, then the above condition is fulfilled automatically. So we reformulate the condition of spherical $t$-design on the antipodal set as follows:
\begin{prop}
A nonempty finite antipodal subset $X\subset S^{n-1}_{m}$ is a spherical $2s+1$-design if the condition 
\begin{align*}
\sum_{x\in X}P(x)=0,  \quad \forall P\in {\rm Harm}_{j}(\R^{n})
\end{align*}
holds for every even integer $2j$ with $2\leq 2j\leq 2s$. 
\end{prop}

A lattice in $\R^{n}$ is a subset $\Lambda \subset \R^{n}$ with the property that there exists a basis $\{e_{1}, \cdots, e_{n}\}$ of $\R^{n}$ such that $\Lambda =\Z e_{1}\oplus \cdots \oplus\Z e_{n}$, i.e., $\Lambda $ consists of all integral linear combinations of the vectors $e_{1}, \cdots, e_{n}$. 
The dual lattice $\Lambda$ is the lattice
\begin{align*}
\Lambda^{\sharp}:=\{y\in \R^{n}\ |\ \langle y| x\rangle \in\Z , \ \forall x\in \Lambda\}. 
\end{align*}
In this paper, we assume that the lattice $\Lambda $ is integral, that is, $\langle x|y\rangle \in\Z$ for all $x$, $y\in \Lambda$. An integral lattice is called even if $\langle x| x\rangle \in 2\Z$ for all $x\in \Lambda$, and it is odd otherwise. An integral lattice is called unimodular if $\Lambda^{\sharp}=\Lambda$. 
For a lattice $\Lambda$ and a positive real number $m>0$, the shell of norm $m$ of $\Lambda$ is defined by 
\[
\Lambda_{m}:=\{x\in \Lambda\ |\ \langle x|x \rangle=m \}=\Lambda\cap (S^{n-1})_{m}.
\]

Let $\mathbb{H} :=\{z\in\C\ |\ \Im z >0\}$ be the upper half-plane. 
\begin{df}
Let $\Lambda$ be the lattice of $\R^{n}$. Then for a polynomial $P$, the function 
\begin{align*}
\Theta _{\Lambda, P} (z):=\sum_{x\in \Lambda}P(x)e^{i\pi z\langle x|x\rangle}
\end{align*}
is called the theta series of $\Lambda $ weighted by $P$. 
\end{df}
\begin{rem}[See Hecke \cite{12}, Schoeneberg \cite{13}, \cite{20}]
$\\$
{\rm (i)}
When $P=1$, we get the classical theta series 
\begin{align*}
\Theta _{\Lambda} (z)=\Theta _{\Lambda, 1} (z)=\sum_{m\ge 0}|\Lambda_{m}|q^{m},\ where\ q=e^{\pi i z}. 
\end{align*}
{\rm (ii)}
The weighted theta series can be written as 
\begin{align*}
\Theta _{\Lambda, P} (z)&=\sum_{x\in \Lambda}P(x)e^{i\pi z\langle x|x\rangle} \\
&=\sum_{m\geq 0}a^{(P)}_{m}q^{m},\ where\ a^{(P)}_{m}:=\sum_{x\in \Lambda_{m}}P(x). 
\end{align*}
\end{rem}

These weighted theta series have been used efficiently for the study of spherical designs which are the shells of Euclidean lattices. (See \cite{9}, \cite{19}, \cite{11}, \cite{2} and \cite{16}. See also \cite{10}.) 

\begin{lem}[cf. \cite{9}, \cite{19}, \cite{2}, Lemma 5]\label{lem:lempache}
Let $\Lambda$ be an integral lattice in $\R^{n}$. Then, for $m>0$, the non-empty shell $\Lambda_{m}$ is a spherical $t$-design if and only if 
\begin{align*}
a^{(P)}_{m}=0\ for\ every\ P\in {\rm Harm}_{2j}(\R^{n}), \ 1\leq 2j\leq t, 
\end{align*}
where $a^{(P)}_{m}$ are the Fourier coefficients of the weighted theta series 
\begin{align*}
\Theta _{\Lambda , P}(z)=\sum_{m\geq 0}a^{(P)}_{m}q^{m}. 
\end{align*}
\end{lem}

The theta series of $\Lambda $ weighted by $P$ is a modular form for some subgroup of $SL_{2}(\R)$. We recall the definition of the modular forms. 

\begin{df}
Let $\Gamma \subset SL_{2}(\R)$ be a Fuchsian group of the first kind and let $\chi$ be a character of $\Gamma$. A holomorphic function $f:\mathbb{H}\rightarrow \C$ is called a modular form of weight $k$ for $\Gamma$ with respect to $\chi$, if the following conditions are satisfied{\rm :} 
\begin{align*}
&{\rm (i)} \quad f\left(\frac{az+b}{cz+d}\right)=\left(\frac{cz+d}{\chi (\sigma)}\right)^{k}f(z)\quad for\ all\ \sigma =
\begin{pmatrix}
a & b\\
c & d
\end{pmatrix}
 \in \Gamma \\
&{\rm (ii)} \quad f(z)\ is\ holomorphic\ at\ every\ cusp\ of\ \Gamma. 
\end{align*}
\end{df}
If $f(z)$ has period $N$, then $f(z)$ has a Fourier expansion at infinity, \cite{3}:
\begin{align*}
f(z)=\sum_{m=0}^{\infty}a_{m} q_{N}^{m},\ q_{N}=e^{2 \pi i z /N}. 
\end{align*}
We remark that for $m<0$, $a_{m}=0$, by condition (ii). A modular form with constant term $a_{0}=0$, is called cusp form. 
We denote by $M_{k}(\Gamma , \chi)$ (resp. $S_{k}(\Gamma , \chi)$) the space of modular forms (resp. cusp forms) with respect to $\Gamma$ with the character $\chi$. 
When $f$ is the normalized eigenform of Hecke operators, p.163, \cite{3}, the Fourier coefficients satisfy the following relations:
\begin{lem}[cf. \cite{3}, Proposition 32, 37, 40, Exercise 2, p.164]\label{lem:lemrama}
Let $f(z)=\sum_{m\geq1}a(m)q^{m} \in S_{k}(\Gamma , \chi)$. 
If $f(z)$ is the normalized eigenform of Hecke operators, then the Fourier coefficients of $f(z)$ satisfy the following relations{\rm :}
\begin{align*}
a(mn)&=a(m)a(n) &(m, n \ {\rm coprime}) \\ 
a(p^{\alpha +1})&=a(p)a(p^{\alpha })-\chi(p) p^{k-1}a(p^{\alpha -1}) &(p {\rm \ a\  prime}). 
\end{align*}
\end{lem}
We set $f(z)=\sum_{m\geq1}a(m)q^{m} \in S_{k}(\Gamma , \chi)$. 
When $\dim S_{k}(\Gamma , \chi)=1$ and $a(1)=1$, then the $f(z)$ is the normalized eigenform of Hecke operators, \cite{3}. So, the coefficients of $f(z)$ have the relations as mentioned in Lemma \ref{lem:lemrama}. 
It is known that 
\begin{align}
|a(p)|<2p^{(k-1)/2} \label{eqn:Deligne}
\end{align}
for all primes $p$. 
Note that this is the Ramanujan conjecture and its generalization, called the Ramanujan-Petersson conjecture for cusp forms which are eigenforms for the Hecke operator. These conjectures were proved by Deligne as a consequence of his proof of the Weil conjectures, p.164, \cite{3}, \cite{14}. 
Moreover the following equation holds, \cite{8}. 
\begin{align}
a(p^{\alpha })=p^{(k-1)\alpha /2}\frac{\sin (\alpha +1)\theta _{p}}{\sin \theta _{p}}, \label{eqn:Lehmer} 
\end{align}
where $2 \cos \theta _{p} = a(p)p^{-(k-1)/2}$.

It is well known that the theta series of $\Lambda \subset \R^{n}$ weighted by harmonic polynomial $P\in {\rm Harm}_{j}(\R^{n})$ is a modular form of weight $n/2+j$ for some subgroup $\Gamma \subset SL_{2}(\R)$. In particular, when the $\deg (P)\geq 1$, the theta series of $\Lambda$ weighted by $P$ is a cusp form. 

For example, we consider the even unimodular lattice $\Lambda$. Then the theta series of $\Lambda $ weighted by harmonic polynomial $P$, $\Theta_{\Lambda, P}(z)$, is a modular form with respect to $SL_{2}(\Z)$. 

\begin{ex}
Let $\Lambda $ be the $E_{8}$-lattice. This is an even unimodular lattice of $\R^{8}$, generated by the $E_{8}$ root system. The theta series is as follows{\rm :}
\begin{align*}
\Theta_{\Lambda}(z)=E_{4}(z)&=1+240\sum_{m=1}^{\infty}\sigma _{3}(m) q^{2m} \\
&=1 + 240 q^2 + 2160 q^4 + 6720 q^6 + 17520 q^8 +\cdots, 
\end{align*}
where $\sigma _{3}(m)$ is a divisor function $\sigma _{3}(m)=\sum_{0<d|m}d^3$. 

For $j=2, 4$ and $6$, the theta series of $\Lambda $ weighted by $P$, $P\in {\rm Harm}_{j}(\R^{8})$ is a weight $6, 8$ and $10$ cusp form with respect to $SL_{2}(\Z)$. However, it is well known that for $k=6, 8$ and $10$, the $\dim S_{k}(SL_{2}(\Z))=0$, that is, $\Theta_{\Lambda , P}(z)=0$. Then by Lemma \ref{lem:lempache}, all the shells of $E_{8}$-lattice are spherical 7-design. 
\end{ex}

For $j=8$, the theta series of $\Lambda $ weighted by $P$ is a weight $12$ cusp form with respect to $SL_{2}(\Z)$. Such a cusp form is uniquely determined up to constant, i.e., is Ramanujan's delta function:
\begin{align*}
\Delta _{24}(z)=q^2\prod_{m\geq 1}(1-q^{2m})^{24}=\sum_{m\geq 1}\tau (z)q^{2m}. 
\end{align*}

The following proposition is due to Venkov, de la Harpe and Pache, \cite{11}, \cite{16}, \cite{2}, \cite{9}. 
\begin{prop}[cf. \cite{2}]\label{prop:proppache}
Let the notation be the same as above. Then the following are equivalent{\rm :} \\
$(\mathrm{i})$\quad $\tau (m)=0${\rm ;} \\ 
$(\mathrm{ii})$\quad $(\Lambda )_{2m}$ is an 8-design.
\end{prop}
It is a famous conjecture of Lehmer that $\tau (m) \neq 0$. So, Proposition \ref{prop:proppache} gives a reformulation of Lehmer's conjecture. 
Lehmer proved in \cite{8} the following theorem. 

\begin{thm}[cf. \cite{8}]
Let $m_{0}$ be the least value of $m$ for which $\tau (m)=0$. Then $m_{0}$ is a prime. 
\end{thm}

These are many attempts to study Lehmer's conjecture (\cite{8}, \cite{21}), but it is difficult to prove and it is still open. 

In this paper, we take the two cases $\Z^{2}$-lattice and $A_{2}$-lattice instead of $E_{8}$-lattice. Then, we consider the analogue of Lehmer's conjecture corresponding to the theta series weighted by some harmonic polynomial $P$. In Section $3$, we show that the $m$-th coefficient of the weighted theta series of $\Z^{2}$-lattice does not vanish when the shell of norm $m$ of those lattices is not an empty set. Or equivalently, we show the following result.
\begin{thm}\label{thm:3.1}
The shells in $\mathbb{Z}^{2}$-lattice are not spherical $5$-designs. 
\end{thm}
Similarly, in Section $4$, we show the following result. 
\begin{thm}\label{thm:4.1}
The shells in $A_{2}$-lattice are not spherical $7$-designs. 
\end{thm}

\section{Preliminaries}
First of all, we list up the famous modular forms needed later. The details of these functions appear in \cite{5}. We use $q=e^{\pi i z}$. 

\begin{align*}
\theta_{2}(z)&=\sum_{m\in \Z+1/2}q^{m^2}=2q^{1/4}(1+q^2+\cdots) &of\ weight\ 1/2, \\
\theta_{3}(z)&=\sum_{m\in \Z}q^{m^2}=1+2q+2q^4+\cdots     &of\ weight\ 1/2, \\
\theta_{4}(z)&=\sum_{m\in \Z}(-q)^{m^2}=1-2q+2q^4+\cdots   &of\ weight\ 1/2, \\
\eta (z)&=q^{1/12}\prod_{m\geq 1}(1-q^{2m})=q^{1/12}(1 - q^2 - q^4 + q^{10}+\cdots)   &of\ weight\ 1/2, \\
\Phi(z)&=\theta_{4}(z)^{4}-\theta_{2}(z)^{4}=1-24q+24q^{2}-96^{3}+\cdots   &of\ weight\ 2,  \hspace{11pt} \\
\Delta_{8}(z)&=\frac{1}{16}\theta_{2}(z)^4\theta_{4}(z)^4=q-8q^2+28q^3+\cdots   &of\ weight\ 4,  \hspace{11pt} \\
\Delta_{12}(z)&=\eta(z)^6\eta(3z)^6=q^2 - 6 q^4 + 9 q^6 + 4 q^8 +\cdots   &of\ weight\ 6.  \hspace{11pt}
\end{align*}

\subsection{The $\Z^{2}$-lattice}      
Let 
\begin{align*}
\Z ^{n}:=\{x=(x_1,\cdots,x_n)\in \R^{n}\ |\ x_{i} \in \Z, i=1,\cdots , n\} 
\end{align*}
be the cubic lattice of rank $n$. It is an odd unimodular lattice. The theta series of $\Z^{n}$ is $\Theta_{\Z ^{n}}(z)=\theta _{3}(z)^{n}$. For example, if we take $n=2$, then 
\begin{align*}
\Theta_{\Z ^{2}}(z)=\theta _{3}(z)^{2}&=\sum_{m=0}^{\infty}r_{2}(m)q^{m} \\
&= 1 + 4q + 4q^2 + 4q^4 + 8q^5 + 4q^8 + 4 q^9 +8 q^{10} + 8 q^{13}+\cdots, 
\end{align*}
where the coefficient $r_{2}(m)$ is the number of ways of writing $m$ as a sum of $2$ squares. 
\begin{lem}[cf. \cite{2}, Lemma 24]
We have 
\begin{align*}
\Theta_{\Z ^{n},P}=
\left\{
\begin{array}{lll}
\theta_{3}^{n} \quad &{\rm if }\ &P=1, \\
0 \quad &{\rm if}\ &P\in {\rm Harm}_{2}(\R^{n}), \\
c_{1}(P)\Delta_{8}\theta_{3}^{n} \quad &{\rm if}\ &P\in {\rm Harm}_{4}(\R^{n}), \\
c_{2}(P)\Phi \Delta_{8}\theta_{3}^{n} \quad &{\rm if}\ &P\in {\rm Harm}_{6}(\R^{n}), 
\end{array} 
\right.
\end{align*}
where $c_{1}$ is a nonzero linear form if and only if $n\geq 2$, and $c_{2}$ is a nonzero linear form if and only if $n\geq 3$. 
\end{lem}

By Lemma \ref{lem:lempache}, For $n\geq 2$, all the nonempty shells of $\Z^{n}$ are 
spherical 3-designs. We consider the case $n=2$. For $P\in {\rm Harm}_{4}(\R^{2})$, 
$\Theta_{\Z ^{2},P}=c_{1}(P)\Delta_{8}\theta_{3}^{2}=\sum_{m>0}a(m)q^{m}$.  

We set 
\begin{align*}
G(2):=
\left<
\begin{pmatrix}
1&2\\
0&1 
\end{pmatrix},\ 
\begin{pmatrix}
0&-1\\
1&0 
\end{pmatrix} 
\right>. 
\end{align*}

Then, the $\Delta_{8}$ and $\theta_{3}^{2}$ are the modular forms 
with respect to $G(2)$ with the character $\chi $, 
\begin{align*}
\left\{
\begin{array}{lll}
\chi \left(\begin{pmatrix}
1&2\\
0&1 
\end{pmatrix}
\right)=1 \\
\chi \left(\begin{pmatrix}
0&-1\\
1&0 
\end{pmatrix}
\right)=i, 
\end{array} 
\right.
\end{align*}
(cf. \cite{2}, Theorem 12, \cite{5}, p.187). 
Hence the $\Delta_{8}\theta_{3}^{2} \in S_{5}(G(2), \chi )$. 

Then it is easy to see that the following proposition holds. 

\begin{prop}[cf. \cite{2}]
Let the notation be the same as above. Then the following assertions are equivalent{\rm :} \\
$(\mathrm{i})$\quad $(\Z^{2})_{m}\neq \emptyset\ and\ a (m)=0${\rm ;} \\ 
$(\mathrm{ii})$\quad $(\Z^{2})_{m}$ is a 5-design.
\end{prop}

In Section \ref{sec:secZ2}, we will prove that $a (m)\neq 0$ if $(\Z^{2})_{m}\neq \emptyset$, hence also show the non existence of the spherical $5$-designs on the shells in $\Z^{2}$-lattice.

\subsection{The $A^{2}$-lattice}                               
Let 
\begin{align*}
A _{n}:=\{x=(x_0,\cdots,x_n)\in \Z^{n+1}\ |\ x_{0}+\cdots +x_{n}=0\} 
\end{align*}
be the $A_{n}$-lattice of rank $n$. It is an even lattice. When $n=2$, the theta series of $A_{2}$-lattice is 
\begin{align*}
\Theta_{A _{2}}(z)&=\theta _{3}(2z)\theta _{3}(6z)+\theta _{2}(2z)\theta _{2}(6z)\\
&=1 + 6 q^2 + 6 q^6 + 6 q^8 + 12 q^{14} + 6 q^{18}+6 q^{24} + 12 q^{26} + 6 q^{32} + 12 q^{38} +\cdots .
\end{align*}

\begin{lem}
We have 
\begin{align*}
\Theta_{A _{2},P}=
\left\{
\begin{array}{lll}
\Theta_{A_2} \quad &if\ P=1, \\
0 \quad &if\ P\in {\rm Harm}_{2}(\R^{2}), \\
0 \quad &if\ P\in {\rm Harm}_{4}(\R^{2}), \\
c_{1}(P)\Delta_{12}\Theta_{A_2} \quad &if\ P\in {\rm Harm}_{6}(\R^{2}), 
\end{array} 
\right.
\end{align*}
where $c_{1}$ is a nonzero linear form. 
\end{lem}
\begin{proof}
First, we define the Fricke group; for any prime p,
\begin{align*}
\Gamma_{0}^{*}(p):=\Gamma_{0}(p) \cup \Gamma_{0}(p)\, W_{p},\quad where\ 
W_{p}:=
\begin{pmatrix}
0        & -1/\sqrt{p} \\
\sqrt{p} & 0 
\end{pmatrix}. 
\end{align*}

Then, The theta series of $A_{2}$-lattice weighted by $P$ are the modular forms with respect to the Fricke group $\Gamma_{0}^{\ast}(3)$ with the character $\chi_{s}$, 
\begin{align*}
\left\{
\begin{array}{lll}
\chi_{s} (A)=\left( \frac{(-3)^{s}}{d}\right) \quad &if\ A=\begin{pmatrix}
a&b\\
c&d 
\end{pmatrix}
\in \Gamma_{0}(3) \\
\chi_{s} (W_{3}) =i^s, 
\end{array} 
\right.
\end{align*}
(cf. \cite{6}, Theorem 3.1). 
In fact, $\oplus_{s\geq 0}M_{s}(\Gamma_{0}^{\ast}(3), \chi_{s}) = \C[\Theta_{A_{2}}, \Delta _{12}]$, \cite{15}.

We take the $P\in {\rm Harm}_{6}(\R^{2})$, 
\begin{align*}
P(x)=(x^6-y^6)-15(x^4y^2-x^2y^4). 
\end{align*}
Then $\Theta _{A_{2}, P}(z)=4 q^{2}+\cdots$. So, $c_{1}$ is not identically zero. 
\end{proof}

By Lemma \ref{lem:lempache}, all the nonempty shells of $A_{n}$ are 
spherical 5-designs. For $P\in {\rm Harm}_{6}(\R^{3})$, 
$\Theta_{A_{2},P}=c_{1}(P)\Delta_{12}\Theta_{A_2}=:\sum_{m>0}a(m)q^{2m}$.  

Then it is easy to see that the following proposition holds. 
\begin{prop}
Let the notation be the same as above. Then the following assertions are equivalent{\rm :} \\
$(\mathrm{i})$\quad $(A_{2})_{m}\neq \emptyset\ and\ a (m)=0${\rm ;} \\ 
$(\mathrm{ii})$\quad $(A_{2})_{2m}$ is a 7-design.
\end{prop}

In Section \ref{sec:secA2}, we will prove that $a (m)\neq 0$ if $(A_{2})_{2m}\neq \emptyset$, hence also show the non existence of the spherical $7$-designs on the shells in $A_{2}$-lattice.

Finally, We collect the results needed later. 

\begin{prop}[cf. \cite{7}]\label{hir:hir}
Let $\sigma _{k}(m)$ be the divisor function $\sigma _{k}(m)=\sum_{0<d|m}d^{k}$. Then following equality holds{\rm :}
\begin{align*}
\frac{1}{16}\theta _{2}^4(z)=\sum_{m=1}^{\infty}\sigma _{1}(2m-1)q^{2m-1}. 
\end{align*}
\end{prop}

\begin{prop}[cf. \cite{2}, p.127]\label{kob:kob}
Let $\chi$ be a Dirichlet character $\mod M$ and $\chi_{1}$ be a primitive Dirichlet character $\mod N$. If $f(z)=\sum_{n=0}^{\infty} a_{n}q^{n}\in M_{k}(\Gamma _{0}(M), \chi)$ and $f_{\chi_{1}}(z)=\sum_{n=0}^{\infty} a_{n}\chi_{1}(n)q^{n}$, then $f_{\chi_{1}}(z)\in M_{k}(\Gamma _{0}(MN^{2}), \chi \chi_{1}^{2})$. If $f$ is cusp form, then so is $f_{\chi_{1}}$. In particular, if $f(z)\in M_{k}(\Gamma_{0}(M))$ and $\chi_{1}$ is a quadratic {\rm(i.e.}, tales values $\pm 1${\rm )}, then $f_{\chi_{1}}(z)\in M_{k}(\Gamma_{0}(MN^{2}))$. 
\end{prop}

\begin{thm}[cf. \cite{4}, Theorem 1]\label{stu:stu}
Let $f(z)$ and $g(z)$ be holomorphic modular forms of weight $k$ with respect to some congruence subgroup $\Gamma$ of $SL_{2}(\Z)$. If $f(z)$ and $g(z)$ have rational integer coefficients and there exists a prime $l$ such that 
\begin{align*}
Ord_{l}(f(z)-g(z))>\frac{k}{12}[SL_{2}(\Z):\Gamma], 
\end{align*}
then the $Ord_{l}(f(z)-g(z))=\infty$. {\rm(i.e.}, $f(z)\equiv g(z) \pmod l. ${\rm )} 
\end{thm}
\section{$\Z ^{2}$-lattice}\label{sec:secZ2}        
We recall the results: 
\begin{align*}
\Theta_{\Z ^{2}}(z)=\theta _{3}(z)^{2}&=\sum_{m=0}^{\infty}r_{2}(m)q^{m}. 
\end{align*}
\begin{lem}\label{lem:lem1}
Assume that $p$ is a prime. If $p \equiv 1 \pmod 4$, then $r_{2}(p)=8$. 
If $p \equiv 3 \pmod 4$, then 
\begin{align*}
r_{2}(p^n)=
\left\{
\begin{array}{lll}
0 \quad &{\rm if}\ &n\ is\ odd, \\
4 \quad &{\rm if}\ &n\ is\ even. 
\end{array} 
\right.
\end{align*}
\end{lem}
\begin{proof}
Denote the number of divisors of $n$ with $d\equiv a \pmod 4$ by $d_{a}(n)$. It is well known that $r_{2}(m)=4(d_{1}(m)-d_{3}(m))$, \cite{5}. Then the results are easy calculation. 
\end{proof}
For $P \in {\rm Harm}_4(\R ^{2})$, 
\begin{align*}
\Theta_{\Z ^{2}, P}(z)&=c(p)\theta _{3}(z)^2\Delta_{8}(z) \\
&=c(p)(q - 4 q^2 + 16 q^4 - 14 q^5 - 64q^8 + 81 q^9+\cdots) \\
&=:c(p)\sum_{m\geq 1}a(m)q^m
\end{align*}
Then $\Theta_{\Z ^{2}, P}(z) \in S_{5}(G(2), \chi )$. 

Here, we difine the following function: 
\begin{align*}
\theta _{3}(2z)^2\Delta_{8}(2z)&=\sum_{m\geq 1}a(m)q^{2m} \\
&=q^2 - 4 q^4 + 16 q^8 - 14 q^{10} - 64q^{16} + 81 q^{18}+\cdots. 
\end{align*}
This is the modular form with respect to $\Gamma _{0}(4)$ 
with the character $\chi_{4}=(-1)^{(d-1)/2}$ 
(cf. \cite{3} Proposition 30, p.145). 

Because of the $\dim S_{5}(\Gamma_{0}(4), \chi_{4} )=1$, \cite{18}, and $a(1)=1$, by Lemma \ref{lem:lemrama}, the coefficients of $\Theta_{\Z ^{2}, P}(z)$ satisfied the following relations: 
\begin{align}
a(mn)&=a(m)a(n) &(m, n \ {\rm coprime}) \label{eqn:Z1} \\ 
a(p^{\alpha +1})&=a(p)a(p^{\alpha })-\chi_{4}(p) p^{4}a(p^{\alpha -1}). &(p {\rm \ a\  prime}) \label{eqn:Z2}
\end{align}
By the equation (\ref{eqn:Deligne}) and (\ref{eqn:Lehmer}), we get the following equations: 
\begin{align}
|a(p)|&<2p^{2} \label{eqn:Z5} \\ 
a(p^{\alpha })&=p^{2\alpha }\frac{\sin (\alpha +1)\theta _{p}}{\sin \theta _{p}}, \label{eqn:Z3}
\end{align}
where $2 \cos \theta _{p} = a(p)p^{-2}$. 

The coefficients $a(m)$ have the following crucial property. 
\begin{lem}\label{lem:lem2}
If $m$ is odd, then $a(m)\equiv \sigma_{1}(m) \pmod 4$, where $\sigma _{1}(m)$ is the divisor function $\sigma _{1}(m):=\sum_{0<d|m}d$. 
\end{lem}
\begin{proof}
Because of the $\theta _{3}(2z)^2\equiv \theta _{4}(2z)^4\equiv 1 \pmod 4$ and Proposition \ref{hir:hir}, 
\begin{align*}
\theta _{3}(2z)^2\Delta_{8}(2z)&=\frac{1}{16}\theta _{3}(2z)^2\theta _{4}(2z)^4\theta _{2}(2z)^4 \\
&\equiv \frac{1}{16}\theta _{2}(2z)^4 \pmod 4 \\
&\equiv \sum_{m=1}^{\infty} \sigma _{1}(2m-1)q^{2(2m-1)} \pmod 4. 
\end{align*}
\end{proof}

\begin{proof}[Proof of Theorem \ref{thm:3.1}]
We will show that the $a(m)\neq 0$ when $(\Z^{2})_{m}\neq \emptyset$. Assume that $m$ is a power of prime, if not we could apply (\ref{eqn:Z1}). We will divide into the three cases. \\
(i) Case $m=2^{\alpha}$: \\
We consider the equation (\ref{eqn:Z2}). 
\begin{align*}
a(2^{n+1})&=a(2)a(2^{n}). 
\end{align*}
Hence we have $a(2^{\alpha})\neq 0$, for $a(2)=-4$. \\
(ii) Case $m=p^{\alpha}$, $p\equiv 3 \pmod 4$: \\
By Lemma \ref{lem:lem1}, $a(p^n)=0$ if $n$ is odd number. Then, equation (\ref{eqn:Z2}) can be written as follows: 
\begin{align*}
a(p^{n+1})= p^4a(p^{n-1}). 
\end{align*}
Thus we get 
\begin{align*}
a(p^{n})= 
\left\{
\begin{array}{lll}
0 \quad &{\rm if}\ &n\ is\ odd, \\
p^{4(n-1)} \quad &{\rm if}\ &n\ is\ even. 
\end{array} 
\right.
\end{align*}
Hence we have $a(p^{\alpha})\neq 0$ when $(\Z^{2})_{m}\neq \emptyset$. \\
(iii) Case $m=p^{\alpha}$, $p\equiv 1 \pmod 4$: 
First of all, we show the following proposition. 
\begin{prop}\label{prop:3.1}
Let $\alpha_0$ be the least value of $\alpha$ for which $a{\rm (}p^{\alpha}{\rm )}=0$. Then $\alpha_0 =1$. 
\end{prop}

\begin{proof}
Assuming the contrary, that is, $\alpha_0 > 1$, so that $a(p)\neq 0$. By (\ref{eqn:Z3}), 
\begin{align*}
a(p^{\alpha_0})=0=p^{2\alpha_0}\frac{\sin (\alpha +1)\theta _{p}}{\sin \theta _{p}}. 
\end{align*}
This shows that $\theta _{p}$ is a real number of the form
$\theta _{p}=\pi k / (1+\alpha_0)$, where $k$ is an integer. Now the number 
\begin{align}
z=2 \cos\theta _{p}=a(p) p^{-2}, \label{eqn:Z4}
\end{align}
being twice the cosine of a rational multiple $2 \pi$, is an algebraic integer. On the other hand $z$ is a root of the obviously irreducible quadratic 
\begin{align}
p^{4}z^{2}-a^{2}(p)=0.  \label{eqn:Z6}
\end{align}
Hence $z$ is a rational integer. By (\ref{eqn:Z5}) and (\ref{eqn:Z4}), we have $z^{2}\leq 3$. Therefore $z^{2}=1, 2$ and $3$. If z=2 (resp. 3), the quadratic (\ref{eqn:Z6}) become $a^{2}(2)= 2 p^{4}$ {\rm (resp.} $a^{2}(3)= 3 p^{4}$). These are impossible because the right hand sides are not square. 
If $z^{2}=1$, the quadratic (\ref{eqn:Z6}) become 
\begin{align}
a(p)=\pm p^{2}. \label{eqn:Z9} 
\end{align}
By Lemma \ref{lem:lem2}, we have 
\begin{align}
a(p)&\equiv \sigma_{1}(p) \pmod 4 \nonumber \\ 
&\equiv p+1 \pmod 4 \nonumber \\
&\equiv 2 \pmod 4 . \label{eqn:Z10}
\end{align}
However, if $p^{2}\equiv 1 \pmod 4$ then the equation (\ref{eqn:Z9}) becomes $a(p)=1$ or $3$. This is a contradiction. 
\end{proof}
So, it is enough to consider the case when $m$ is a prime. 
By (\ref{eqn:Z10}), $a(p)\equiv 2 \pmod 4$, so, we have $a(p)\neq 0$. This completes the proof of Theorem \ref{thm:3.1}. 
\end{proof}

\section{$A _{2}$-lattice}\label{sec:secA2}
We recall the results. 
\begin{align*}
\Theta_{A _{2}}(z)&=\theta _{3}(2z)\theta _{3}(6z)+\theta _{2}(2z)\theta _{2}(6z)\\
&=:\sum_{m=0}^{\infty}N(m)q^{2m}. 
\end{align*}
\begin{lem}[cf. \cite{5}, p.112]\label{lem:lemA1}
\begin{align*}
N(3^{\alpha})&=6, \hspace{45pt} for\ all\ a\geq 0, \\
N(p^{\alpha})&=6(\alpha +1), \hspace{12pt} for\ p \equiv 1 \pmod 3, \\
N(p^{\alpha})&=
\left\{
\begin{array}{lll}
0 \quad \hspace{15pt}&for\ p \equiv 2 \pmod 3, \alpha\ is\ odd, \\
6 \quad &for\ p \equiv 2 \pmod 3, \alpha\ is\ even.
\end{array} 
\right.
\end{align*}
\end{lem}
For $P \in {\rm Harm}_6(\R ^{2})$, 
\begin{align*}
\Theta_{A _{2}, P}(z)&=c(p)\Theta _{A_{2}}(z)\Delta_{12}(z) =:\sum_{m\geq 1}a(m)q^{2m}, \\
\Delta_{12}(z)&=(\eta (z)\eta (3z))^{6}(z) =:\sum_{m\geq 1}b(m)q^{2m}, \\
E_{6}(z)&=1-504\sum_{m\geq 1}\sigma_{5}(m)q^{2m}, 
\end{align*}
where $E_{6}(z)$ is the Eisenstein series of weight $6$ with respect to the group $SL_{2}(\Z)$ and, $\sigma _{5}(m)$ is the divisor function $\sigma _{5}(m):=\sum_{0<d|m}d^{5}$. 
As we saw as above, $\Theta_{A _{2}, P}(z) \in S_{7}(\Gamma _{0}^{\ast}(3), \chi_{7})$, hence $\Theta_{A _{2}, P}(z) \in S_{7}(\Gamma _{0}(3), \chi_{7})$. 
Because of the $\dim S_{7}(\Gamma _{0}(3), \chi_{7})=1$, \cite{18}, and $a(1)=1$, by Lemma \ref{lem:lemrama}, the coefficients of $\Theta_{A _{2}, P}(z)$ satisfies the following relations: 
\begin{align}
a(mn)&=a(m)a(n) &(m, n \ {\rm coprime}) \label{eqn:A1} \\ 
a(p^{\alpha +1})&=a(p)a(p^{\alpha })-\chi(p) p^{6}a(p^{\alpha -1}). &(p {\rm \ a\  prime}) \label{eqn:A2}
\end{align}
By the equation (\ref{eqn:Deligne}) and (\ref{eqn:Lehmer}), we get the following equations: 
\begin{align}
|a(p)|&<2p^{3} \label{eqn:A3} \\ 
a(p^{\alpha })&=p^{3\alpha }\frac{\sin (\alpha +1)\theta _{p}}{\sin \theta _{p}}, \label{eqn:A4}
\end{align}
where $2 \cos \theta _{p} = a(p)p^{-3}$. 

The following lemma is useful later. 
\begin{lem}\label{lem:lemA2}
For $m\not\equiv 0 \pmod 3$, $a(m)\equiv \sigma_{5}(m) \pmod 3$. 
\end{lem}
\begin{proof}
By Lemma \ref{lem:lemA1}, we remark that $\Theta _{A_{2}}(z) \equiv 1 \pmod 3$. So, we have $a(m)\equiv b(m) \pmod 3$. Next we consider the following function: 
\begin{align*}
\Delta_{12}(z)+\frac{E_{6}(z)}{504}=\frac{1}{504}+\sum_{m\geq 0}(b(m)-\sigma_{5}(m))q^{2m}=:\sum_{m\geq 0}c(m)q^{2m}. 
\end{align*}
This is the modular form of $\Gamma _{0}(3)$. Now we apply Proposition \ref{kob:kob} and construct a modular form of $\Gamma _{0}(27)$ by
\begin{align*}
\sum_{m\geq 0}\Big(\frac{m}{3}\Big)c(m)q^{2m}=39 q^4  - 1053 q^8 + 3120 q^{10} - 16848 q^{14} + \cdots . 
\end{align*}
One finds that the first $(6/12)[SL_{2}(\Z):\Gamma_{0}(27)]+1=19$ terms are multiple of $3$ which completes the proof by an immediate application of Theorem \ref{stu:stu}. 
\end{proof}

\begin{lem}\label{lem:lemA3}
For $m\equiv 1 \pmod 3$ and $m$ is odd, $a(m)\equiv \sigma_{1}(m) \pmod 2$. 
\end{lem}
\begin{proof}
By Lemma \ref{lem:lemA1}, we remark that $\Theta _{A_{2}}(z) \equiv \theta _{3}^4(z) \equiv \theta _{4}^4(z) \equiv  1 \pmod 2$. So, we have $a(m)\equiv b(m) \pmod 2$ and we take the following function: 
\begin{align*}
\sum_{m\geq 1}c(m)q^{2m}:=\frac{1}{16}\theta _{3}^4(2z) \theta _{4}^4(2z) \theta _{2}^4(2z)&\equiv \frac{1}{16}\theta _{2}^4(2z) \pmod 2 \\
&\equiv \sum_{m\geq 1}\sigma _{1}(2m-1)q^{2(2m-1)} \pmod 2. 
\end{align*}
This is a modular form of $\Gamma _{0}(4)$, \cite{3}. 
So, for $m\equiv 1 \pmod 2$, $c(m)\equiv \sigma _{1}(m)$. Namely, 
it is enough to show that for $m\equiv 1 \pmod 2$, $b(m)\equiv c(m)$. 

Next we consider the following function: 
\begin{align*}
\sum_{m\geq 1}d(m)q^{2m}&:=\Delta_{12}(z)-\frac{1}{16}\theta_{3}^4(2z)\theta_{4}^4(2z)\theta_{2}^4(2z)\\ &\equiv \Delta_{12}(z)-\sum_{m\geq 1}\sigma _{1}(2m-1)q^{2(2m-1)} \pmod 2 
\end{align*}
This is the modular form of $\Gamma _{0}(12)$. Now we apply Proposition \ref{kob:kob} and construct a modular form of $\Gamma _{0}(108)$ by
\begin{align*}
\sum_{m\geq 1}\Big(\frac{m}{3}\Big)d(m)q^{2m}=6 q^4 + 4 q^8 -48 q^{14} - 168 q^{16} - 36 q^{20} + \cdots . 
\end{align*}
One finds that the first $(6/12)[SL_{2}(\Z):\Gamma_{0}(108)]+1=109$ terms are multiple of $2$ which completes the proof by an immediate application of Theorem \ref{stu:stu}. 
\end{proof}

\begin{proof}[Proof of Theorem \ref{thm:4.1}]
We will show that the $a(m)\neq 0$ when $(A_{2})_{m}\neq \emptyset$. Assume that $m$ is a power of prime, if not we could apply (\ref{eqn:A1}). We will divide into the three cases. \\
(i) Case $m=3^{\alpha}$: \\
We consider the equation (\ref{eqn:A2}). 
\begin{align*}
a(3^{n+1})&=a(3)a(3^{n}). 
\end{align*}
Hence we have $a(3^{\alpha})\neq 0$, for $a(3)=-27$. \\
(ii) Case $m=p^{\alpha}$, $p\equiv 2 \pmod 3$: \\
By Lemma \ref{lem:lemA1}, $a(p^n)=0$ if $n$ is odd number. Then, equation (\ref{eqn:A2}) can be written as follows: 
\begin{align*}
a(p^{n+1})= p^6a(p^{n-1}). 
\end{align*}
Thus we get 
\begin{align*}
a(p^{n})= 
\left\{
\begin{array}{lll}
0 \quad &{\rm if}\ &n\ is\ odd, \\
p^{6(n-1)} \quad &{\rm if}\ &n\ is\ even. 
\end{array} 
\right.
\end{align*}
Hence we have $a(p^{\alpha})\neq 0$ when $(A_{2})_{m}\neq \emptyset$. \\
(iii) Case $m=p^{\alpha}$, $p\equiv 1 \pmod 3$: 
First of all, we show the following proposition. 
\begin{prop}\label{prop:4.1}
Let $\alpha_0$ be the least value of $\alpha$ for which $a{\rm (}p^{\alpha}{\rm )}=0$. Then $\alpha_0 =1$. 
\end{prop}

\begin{proof}
Assuming the contrary, that is, $\alpha_0 > 1$, so that $a(p)\neq 0$. By (\ref{eqn:A3}), 
\begin{align*}
a(p^{\alpha_0})=0=p^{2\alpha_0}\frac{\sin (\alpha +1)\theta _{p}}{\sin \theta _{p}}. 
\end{align*}
This shows that $\theta _{p}$ is a real number of the form $\theta _{p}=\pi k / (1+\alpha_0)$, where k is an integer. Now the number 
\begin{align}
z=2 \cos \theta _{p}=a(p) p^{-3}, \label{eqn:A5}
\end{align}
being twice the cosine of a rational multiple $2 \pi$, is an algebraic integer. On the other hand $z$ is a root of the obviously irreducible quadratic 
\begin{align}
p^{6}z^{2}-a^{2}(p)=0.  \label{eqn:A6}
\end{align}
Hence $z$ is a rational integer. By (\ref{eqn:A3}) and (\ref{eqn:A5}), we have $z^{2}\leq 3$. Therefore $z^{2}=1, 2 $ and $ 3$. If z=2 (resp. 3), (\ref{eqn:A6}) become $a^{2}(2)= 2 p^{6}$ {\rm (resp.} $a^{2}(3)= 3 p^{6}$$)$. These are impossible because the right hand sides are not square. 
If $z^{2}=1$, we have 
\begin{align}
a(p)=\pm p^{3}. \label{eqn:A9} 
\end{align}
By Lemma \ref{lem:lemA2} and \ref{lem:lemA3}, 
\begin{align}
a(p)&\equiv \sigma_{5}(p) \pmod 3 \nonumber \\ 
&\equiv p^5+1 \pmod 3 \nonumber \\
&\equiv 2 \pmod 3 , \label{eqn:A10} \\
a(p)&\equiv \sigma_{1}(p) \pmod 2 \nonumber \\ 
&\equiv p+1 \pmod 2 \nonumber \\
&\equiv 0 \pmod 2 . \label{eqn:A11}
\end{align}
By (\ref{eqn:A9}), $a(p)\equiv 1 $ or $ 2 \pmod 3$. In the first case $a(p)\equiv 1 \pmod 3$, this is a contradiction to (\ref{eqn:A10}). In the second case $a(p)\equiv 2 \pmod 3$, this is a contradiction to (\ref{eqn:A11}). So, the proof is completed. 
\end{proof}
So it is enough to consider the case when $m$ is a prime. 
By (\ref{eqn:A10}), $a(p)\equiv 2 \pmod 3$, so, we have $a(p)\neq 0$. This completes the proof of Theorem \ref{thm:4.1}. 

\end{proof}

\section{Concluding Remarks}
(1) In the last part of the proof of Theorem \ref{thm:3.1} (resp. Theorem \ref{thm:4.1}), after we obtain Proposition \ref{prop:3.1} (resp. Proposition \ref{prop:4.1}), we can directly show that the shells 
$(\Z^2)_p$ (resp. $(A_2)_{2p}$) are not spherical $5$-designs (resp. $7$-designs).
This gives an alternative approach to the proof of Theorem \ref{thm:3.1} 
(resp. Theorem \ref{thm:4.1}). \\

(2) It is interesting to note that no spherical $12$-design among the shells 
of any Euclidean lattice (of any dimension) is known. It is an interesting open problem 
to prove or disprove whether these exists any $12$-design which is a shell of a Euclidean lattice. \\

(3) Responding to the author's request, Junichi Shigezumi performed computer 
calculations to determine whether there are spherical $t$-designs for bigger $t$, in the $2$ 
and $3$ dimensional cases. His calculation shows that among the shells of 
integral lattices in dimension $2$ (with relatively small discriminant 
and small norms), there are only $5$-designs. 
That is, no $6$-designs were found. (So far, all examples of such 5-designs 
are vertices of a regular $6$-gon, although they are the shells 
of many different lattices). In the $3$ dimensional case, 
all examples are only $3$-designs. No $4$-designs which are shells 
of a lattice were found. It is an interesting open problem 
whether this is true in general for dimensions $2$ and $3$. 

\begin{center}
{\bf Acknowledgment}
\end{center}
The second author is supported by JSPS Research Fellowship. The authors thank Junichi Shigezumi and Hiroshi Nozaki for their helpful discussions and computations on this research. 

\begin{appendix}

\end{appendix}

\end{document}